\documentclass[3p,10pt,a4paper,twoside,fleqn,sort&compress]{filomat}
\usepackage{amssymb,amsmath,latexsym}
\usepackage[varg]{pxfonts}

\newtheorem{theorem}{Theorem}[section]

\newtheorem{corollary}[theorem]{Corollary}

\newtheorem{definition}[theorem]{Definition}
\newtheorem{example}[theorem]{Example}

\newtheorem{lemma}[theorem]{Lemma}

\newtheorem{proposition}[theorem]{Proposition}
\newtheorem{remark}[theorem]{Remark}

\begin{document}

%

\title{Monotonicity of sets in Hadamard spaces from polarity point of view}

\author[affil1]{Ali Moslemipour}
\ead{ali.moslemipour@gmail.com}
\author[affil2]{Mehdi Roohi}
\ead{m.roohi@gu.ac.ir}
\address[affil1]{Department of Mathematics,
Aliabad Katoul Branch,  Islamic Azad University,
Aliabad Katoul, Iran}
\address[affil2]{Department of Mathematics, 
Faculty of Sciences, Golestan University,
Gorgan, Iran}
\newcommand{\AuthorNames}{A. Moslemipour, M. Roohi}

\newcommand{\FilMSC}{Primary 47H05; Secondary 47H04, 49J40}
\newcommand{\FilKeywords}{(Monotone sets, maximal monotone sets, polarity, flat Hadamard space,
$\mathbb{R}$-tree)}

\begin{abstract}
This paper is devoted to introduce and investigate the notion of monotone sets in Hadamard spaces. First, flat Hadamard spaces are introduced and investigated. It is shown that an Hadamard space $X$ is flat if and only if $X\times X^{\scalebox{0.55}{$\lozenge$}}$ has $\mathcal{F}_l$-property, where $X^{\scalebox{0.55}{$\lozenge$}}$ is the linear dual of $X$. Moreover, monotone and maximal monotone sets are introduced and also monotonicity from polarity point of view is considered. Some characterizations of (maximal) monotone sets, specially based on polarity, are given. Finally, it is proved that any maximal monotone set is sequentially $bw\times${$\|\cdot\|_{\scalebox{0.5}{$\lozenge$}}$}-closed in $X\times X^{\scalebox{0.55}{$\lozenge$}}$.
\end{abstract}

\maketitle

\makeatletter
\renewcommand\@makefnmark%
{\mbox{\textsuperscript{\normalfont\@thefnmark)}}}
\makeatother

\section{Introduction}

\def\smath#1{\text{\scalebox{0.8}{$#1$}}}
\def\sfrac#1#2{\smath{\frac{#1}{#2}}}
\def\smath#1{\text{\scalebox{0.85}{$#1$}}}
\def\sfrac#1#2{\smath{\frac{#1}{#2}}}
\def\sbig#1{\smath{\big#1}}
\def\Big#1{\smath{\bigg#1}}
\def\sBig#1{\smath{\Big#1}}
\def\Big#1{\smath{\Bigg#1}}
\def\bmath#1{\text{\scalebox{1.1}{$#1$}}}
\def\Bmath#1{\text{\scalebox{1.2}{$#1$}}}
\def\bbmath#1{\text{\scalebox{1.3}{$#1$}}}
\def\BBmath#1{\text{\scalebox{1.5}{$#1$}}}
\def\blangle#1{\bmath{\langle#1}}
\def\Blangle#1{\Bmath{\langle#1}}
\def\bblangle#1{\bbmath{\langle#1}}
\def\BBrangle#1{\BBmath{\rangle#1}}
\def\brangle#1{\bmath{\rangle#1}}
\def\Brangle#1{\Bmath{\rangle#1}}
\def\bbrangle#1{\bbmath{\rangle#1}}

\newcommand{\card}{\mathrm{card}}
\newcommand{\cat}{\mathrm{CAT}}
\newcommand{\Lip}{\mathrm{Lip}}
\newcommand{\spa}{\mathrm{span}}
\newcommand{\Ran}{\mathrm{Range}}
\newcommand{\Dom}{\mathrm{Dom}}
\newcommand{\dom}{\mathrm{dom}}

\newcommand*{\LargerCdot}{\raisebox{-0.95ex}{\scalebox{2.5}{$\cdot$}}}
\newcommand*{\loze}{{\scalebox{0.55}{$\lozenge$}}}
\newcommand*{\loz}{{\scalebox{0.5}{$\lozenge$}}}

In this section, we collect some fundamental definitions and  general notations of Hadamard spaces that will be used throughout of this paper. For background materials on Hadamard spaces, we refer to the standard texts and literatures such as \cite{Bacak2014,BridsonHaefliger,BergNikolaev,KakavandiAmini}.

Given a metric space $(X,d)$, a \textit{geodesic path} from $x \in X$ to $y \in X$ is a map $c:[0,1] \rightarrow X$ such that $c(0)=x$, $c(1)=y$ and $d(c(t),c(s))=|t-s|d(x,y)$, for each $t,s \in [0,1]$.
The image of $c$ is called a \textit{geodesic segment} joining $x$ and $y$.
A metric space $(X,d)$ is a \textit{geodesic space} if for each $x,y \in X$, there exists a geodesic path from $x$ to $y$.
Also, a geodesic space $(X,d)$ is called \textit{uniquely geodesic space} if for every $x, y\in X$ there exists a unique geodesic from $x$ to $y$.

A geodesic space $(X,d)$ is called a \textit{$\cat(0)$ space} if we have:
\begin{align}\label{cn}
d(z,c(t))^2 \leq (1-t)d(z,x)^2 +td(z,y)^2-t(1-t)d(x,y)^2,
\end{align}
for each geodesic path $c:[0,1] \rightarrow X$ from $x$ to $y$, each $z \in X$ and each $t \in [0,1]$. Inequality (\ref{cn}) is called the \textit{CN-inequality}.  It is known {\rm{\rm\cite[Theorem 1.3.3]{Bacak2014}}} that  $\cat(0)$ spaces are uniquely geodesic spaces. Furthermore, the image of unique geodesic path $c:[0,1] \rightarrow X$  from  initial point $x$ to terminal point $y$ is denoted by $[x,y]$, i.e., $c([0,1])=[x,y]$. In this case, for each $z \in [x,y]$, we write $z=(1-t)x \oplus ty$ and say that $z$ is a \textit{convex combination} of $x$ and $y$. A complete $\cat(0)$ space is called \textit{Hadamard space}. Basic examples of Hadamard spaces are: Hilbert spaces, Hadamard manifolds, Euclidean Buildings and $\mathbb{R}$-trees, (see {\rm{\rm\cite[Chapter II.1, 1.15]{BridsonHaefliger}}} for many other examples).

In 2008, Berg and Nikolaev \cite{BergNikolaev} introduced the concept of \textit{quasilinearization} in abstract metric spaces. Ahmadi Kakavandi and Amini \cite{KakavandiAmini} defined the \textit{dual space} for an Hadamard space $(X,d)$  by using the concept of quasilinearization of it. More precisely, let $X$ be an Hadamard space. For each $x, y\in X$, the ordered pair $(x, y) \in X^2$ is denoted by $\overrightarrow{xy}$ and will be called a \textit{bound vector}. For each $x \in X$, the \textit{zero bound vector} at $x\in X$ will be written as $\mathbf{0}_x:=\overrightarrow{xx}$. We identify two bound vectors $-\overrightarrow{xy}$ and $\overrightarrow{yx}$. The bound vectors $\overrightarrow{xy}$ and $\overrightarrow{uz}$ are called \textit{admissible} if $y=u$. The operation of addition of two admissible bound vectors $\overrightarrow{xy}$ and $\overrightarrow{yz}$ is defined by  $\overrightarrow{xy}+\overrightarrow{yz}:=\overrightarrow{xz}$.
The \textit{quasilinearization map} is defined by
\begin{align}\label{inb}
\langle \cdot,\cdot\rangle &: X^2 \times X^2 \rightarrow \mathbb{R},\\ \nonumber
\langle \overrightarrow{ab}, \overrightarrow{cd} \rangle & :=\frac{1}{2}\big(d(a,d)^2+d(b,c)^2-d(a,c)^2-d(b,d)^2\big), ~a, b, c, d,\in X.
\end{align}
\begin{remark}\label{inprob}
For each $x,y,z,u,v \in X$ we have:\\[-1mm]

{\rm(i)} $\langle \overrightarrow{xy}, \overrightarrow{uv} \rangle=\langle \overrightarrow{uv}, \overrightarrow{xy} \rangle$,\\[-1mm]

{\rm(ii)} $\langle \overrightarrow{xy}, \overrightarrow{uv} \rangle=-\langle \overrightarrow{yx}, \overrightarrow{uv} \rangle$,\\[-1mm]

{\rm(iii)} $\langle \overrightarrow{xy}, \overrightarrow{uv} \rangle=\langle \overrightarrow{xz}, \overrightarrow{uv} \rangle+\langle \overrightarrow{zy}, \overrightarrow{uv} \rangle$.
\end{remark}

In \cite[Corollary 3]{BergNikolaev} based on the \textit{Cauchy-Schwarz inequality} (see \eqref{csineq}), $\cat(0)$ spaces have been characterized. More precisely, a geodesic space $(X,d)$ is a $\cat(0)$ space if and only if for each $a, b, c, d\in X$ we have:
\begin{equation} \label{csineq}
\langle \overrightarrow{ab}, \overrightarrow{cd} \rangle \leq d(a,b)d(c,d).
\end{equation}
Consider the map
\begin{align*}
\Theta &:\mathbb{R} \times X^2 \rightarrow C(X,\mathbb{R})\\
& (t,a,b) \mapsto \Theta (t,a,b)x=t \langle \overrightarrow{ab}, \overrightarrow{ax} \rangle, ~a,b,x \in X ,~t \in \mathbb{R},
\end{align*}
where $ C(X,\mathbb{R}) $ is the space of all continuous real-valued functions on $\mathbb{R} \times X^2 $. It follows from the Cauchy-Schwarz inequality \eqref{csineq} that $  \Theta (t,a,b) $ is a Lipschitz function with Lipschitz semi-norm
\begin{equation}\label{lips}
L(\Theta (t,a,b) )=|t|d(a,b) ,~a,b \in X,~t \in \mathbb{R}.
\end{equation}
Recall that the \textit{Lipschitz semi-norm} on $C(X,\mathbb{R})$ is defined by
\begin{align*}
L&:C(X,\mathbb{R}) \rightarrow \mathbb{R}\\
\varphi & \mapsto \sup\Big\{ \frac{\varphi(x)-\varphi(y)}{d(x,y)}:x,y \in X, x\neq y\Big\}.
\end{align*}
The Lipschitz semi-norm (\ref{lips}) induces a \textit{pseudometric} $D$ on $\mathbb{R} \times X^2$, which is defined by
\[D((t,a,b),(s,c,d))=L(\Theta (t,a,b)-\Theta (s,c,d)),~a,b,c,d \in X, \,t,s \in \mathbb{R}.\]

The pseudometric space $ (\mathbb{R} \times X^2,D)$ can be considered as a subspace of the pseudometric space of all real-valued Lipschitz functions $(\Lip(X,\mathbb{R}), L)$. \mbox{Accordance} to \cite[Lemma 2.1]{KakavandiAmini}), we have:
\[ D((t,a,b),(s,c,d))=0 ~\text{if and only if} ~t\langle \overrightarrow{ab},\overrightarrow{xy} \rangle=s\langle \overrightarrow{cd},\overrightarrow{xy} \rangle~\text{for all}~ x,y \in X .\]
It is easily seen that $D$ induces an equivalence relation on $\mathbb{R} \times X^2$. Indeed,  the equivalence class of $(t,a,b)\in\mathbb{R} \times X^2$ is given by
\[[t\overrightarrow{ab}]=\big\{s\overrightarrow{cd}: D((t,a,b),(s,c,d))=0\big\}. \]

The set of all equivalence classes equipped with the metric $D$, defined by
\[D([t\overrightarrow{ab}], [s\overrightarrow{cd}]):=D((t,a,b),(s,c,d)),\]
 is called the \textit{dual space} of the Hadamard space $X$, and is denoted by $X^*$. By using the definition of equivalence classes, we get $[\overrightarrow{aa}]=[\overrightarrow{bb}]$ for each $a,b \in X$.  In general, $X^*$ acts on $X^2$ by
\begin{equation}\label{efxonab}
\langle x^*,\overrightarrow{xy} \rangle =t \langle \overrightarrow{ab}, \overrightarrow{xy} \rangle,~\text{where}~x^*=[t\overrightarrow{ab}] \in X^*, a,b\in X, t\in\mathbb{R}~\text{and}~ \overrightarrow{xy} \in X^2.
\end{equation}

Throughout this paper, we use the following notation:
\[\big\langle \sum_{i=1}^n\alpha_ix_i^*,\overrightarrow{xy} \big\rangle:=\sum_{i=1}^n\alpha_i\big\langle x_i^*,\overrightarrow{xy}\big\rangle,~\alpha_i \in \mathbb{R}, x_i^* \in X^*, x, y\in X, n\in \mathbb{N}.\]
In \cite{ChaipunyaKumam}, Chaipunya and Kumam introduced the concept of  \textit{linear dual space} of an Hadamard space $(X,d)$, as follows:
\[X^{\loze}=\Big\{\sum_{i=1}^n\alpha_i x_i^* :\alpha_i \in \mathbb{R}, x_i^* \in X^*, n \in \mathbb{N}\Big\}.\]
Indeed, $X^\loze={\spa}\, X^*$\!. The zero element of $X^\loze$ is denoted by $\mathbf{0}_{X^\loz}:=[t\overrightarrow{aa}]$, where  $a \in X$ and $t \in \mathbb{R}$. One can see that the evaluation $\langle \mathbf{0}_{X^\loz}, \cdot \rangle$ vanishes on $X^2$.
It is worth mentioned that  $X^\loze$ is a normed space with the norm $\|x^\loz\|_{\loz}=L(x^\loz)$, for all $x^\loz \in X^\loze$. Indeed:
\begin{proposition} {\rm\cite[Proposition 3.5]{ZamaniRaeisi}}\label{dg} Let $X$ be an Hadamard space with the linear dual space $X^\loze$ and let $x^\loz \in X^\loze$ be arbitrary. Then
\begin{align*}
\|x^\loz \|_{\loz}:=\sup\Big\{\sfrac{\big|\Blangle x^\loz,\overrightarrow{ab}\Brangle-
\Blangle x^\loz,\overrightarrow{cd}\Brangle~\!\!\big|}{d(a,b)+d(c,d)}: a,b,c,d \in X, (a,c)\neq (b,d) \Big\},
\end{align*}
is a norm on $X^\loze$. In particular, $\|[t\overrightarrow{ab}]\|_{\loz}=|t|d(a,b)$.
\end{proposition}

\begin{remark} In view of \eqref{efxonab} and Remark \ref{inprob}(iii), for each $a,b,w \in X$ and each $x^{*}=[t\overrightarrow{uv}]$ with $t \in \mathbb{R}$,  $u,v \in X$, we get:
\begin{align*}
\langle x^{*},\overrightarrow{ab} \rangle
&=t\langle \overrightarrow{uv},\overrightarrow{ab} \rangle
=t\langle \overrightarrow{uv},\overrightarrow{aw}+\overrightarrow{wb} \rangle=t\big(\langle \overrightarrow{uv},\overrightarrow{aw} \rangle+ \langle \overrightarrow{uv},\overrightarrow{wb} \rangle\big)= \langle x^{*},\overrightarrow{aw} \rangle+\langle x^{*},\overrightarrow{wb}\rangle.
\end{align*}
Also, let $x^\loz \in X^\loze$ be fixed and arbitrary. Then there are $n\in\mathbb{N}$, $\alpha_1, \alpha_2, \ldots, \alpha_n\in\mathbb{R}$ and $x^{*}_1, x^{*}_2, \ldots, x^{*}_n\in X^*$ such that $x^\loz=\sum_{i=1}^n\alpha_i x_i^{*}$. Then
\begin{align*}
\langle x^\loz,\overrightarrow{ab} \rangle &=\Big\langle \sum_{i=1}^n\alpha_i x_i^{*},\overrightarrow{ab} \Big\rangle 
= \sum_{i=1}^n\alpha_i \langle x_i^{*},\overrightarrow{ab} \rangle \\
&= \sum_{i=1}^n\alpha_i \big(\langle x_i^{*},\overrightarrow{aw} \rangle + \langle x_i^{*},\overrightarrow{wb} \rangle\big)\\
&= \sum_{i=1}^n\alpha_i \langle x_i^{*},\overrightarrow{aw} \rangle +\sum_{i=1}^n\alpha_i \langle x_i^{*},\overrightarrow{wb} \rangle\\
&=\Big\langle \sum_{i=1}^n\alpha_i x_i^{*},\overrightarrow{aw} \Big\rangle+\Big\langle \sum_{i=1}^n\alpha_i x_i^{*},\overrightarrow{wb} \Big\rangle\\
&=\langle x^\loz,\overrightarrow{aw} \rangle+\langle x^\loz,\overrightarrow{wb} \rangle.
\end{align*}
Therefore, $\langle x^\loz,\overrightarrow{ab} \rangle=\langle x^\loz,\overrightarrow{aw} \rangle+\langle x^\loz,\overrightarrow{wb} \rangle$.
\end{remark}
\begin{definition} {\rm{\rm\cite[Definition 2.4]{KakavandiAmini}}}{\rm\,Let $\{x_n\}$ be a sequence in an Hadamard space $X$. The sequence $\{x_n\}$ is said to be \textit{weakly convergent} to $x\in X$, denoted by $x_n\overset{w}\longrightarrow x$, if $\lim_{n \rightarrow \infty} \langle \overrightarrow{xx_n},\overrightarrow{xy} \rangle =0$, for all $y \in X$.
}\end{definition}
One can easily see that convergence in the metric implies weak convergence.
\begin{proposition} {\rm{\rm\cite[Proposition 3.6]{ZamaniRaeisi}}} \label{loz1}Let $\{x_n\} $ be a bounded sequence in  an Hadamard space $(X,d)$ with the linear dual space $X^\loze$ and let $\{x^\loz_n\}$ be a sequence in $X^{\loze}$. If  $\{x_n\}$ is  weakly convergent to $x\in X$ and $x^{\loz}_n\xrightarrow{\|\cdot\|_\loz}x^\loz$, then $\langle x^\loz_n,\overrightarrow{x_nz}\rangle \rightarrow \langle x^\loz,\overrightarrow{xz}\rangle $,  for all $z\in X$.
\end{proposition}
\begin{definition}{\rm Let $X$ be an Hadamard space with the linear dual space $X^\loze$.
\begin{enumerate}
\item[(i)] A sequence $\{x_n\} \subseteq X$ is \textit{$bw$-convergent} to $x\in X$, if $\{x_n\}$ is bounded and $x_n\overset{w}\longrightarrow x$. In this case, we write $x_n\overset{bw}\longrightarrow x$.

\item[(ii)] A sequence $\{(x_n,x^{\loz}_n)\}\subseteq X \times X^\loze$ is \textit{ $bw\times\|\cdot\|_\loz$-convergent}  to $(x, x^\loz)\in X\times X^\loze$, if $x_n\overset{bw}\longrightarrow x$ and $x^{\loz}_n\xrightarrow{\|\cdot\|_\loz}x^\loz$. In this case, we write $(x_n, x^\loz_n)\xrightarrow{bw\times\|\cdot\|_\loz} (x, x^\loz)$.
\item[(iii)]  A subset $M$ of  $X \times X^\loze$ is called \textit{sequentially $bw\times\|\cdot\|_\loz$-closed} if the limit of every $bw\times\|\cdot\|_\loz$-convergent sequence $\{(x_n,x^{\loz}_n)\}\subseteq M$  is in $M$.
\item[(iv)] A mapping $\varphi : X\times X^\loze \rightarrow\, ]-\infty,\infty]$ is  \textit{sequentially $bw\times\|\cdot\|_\loz$-continuous at $(x, x^\loz)\!\in\!X\!\times\!X^\loze$} if for every $\{(x_n,x^{\loz}_n)\}\!\subseteq\!X\!\times\!X^\loze$, with $(x_n, x^\loz_n)\!\xrightarrow{bw\times\|\cdot\|_\loz}\!(x, x^\loz)$ we have $\varphi(x_n,x^{\loz}_n)\rightarrow\varphi(x,x^{\loz})$. Also, $\varphi$ is \textit{sequentially $bw\times\|\cdot\|_\loz$-continuous}  if it is sequentially $bw\times\|\cdot\|_\loz$-continuous at each point of $X\times X^\loze$.
\end{enumerate}
}\end{definition}

\begin{definition}{\rm Let $X$ be an Hadamard space with linear dual space $X^\loze$. For an arbitrary point $p \in X$, we define the \textit{$p$-coupling function} of the dual pair $(X,X^\loze)$ by \[\pi_p:X \times X^\loze \rightarrow \mathbb{R};~(x, x^{\loz})\mapsto \langle x^{\loz}, \overrightarrow{px}\rangle.\]
}\end{definition}
This function is useful in the formulation of some basic results of the monotone sets in Hadamard spaces (see  Proposition \ref{mumon}).
\begin{proposition}\label{pictsconvex} Let $X$ be an Hadamard space with linear dual space $X^\loze$ and $p\in X$. Then the following hold:
\begin{enumerate}
\item[(i)] $\pi_p$ is sequentially $bw\times\|\cdot\|_\loz$-continuous and hence it is continuous.
\item[(ii)] $\pi_p$ is  linear with respect to it's second variable.
 \end{enumerate}
\end{proposition}
\begin{description}
\item[{\rm\textit{Proof.} (i)}]  Let  $\{ (x_n,x^{\loz}_n)\}\subseteq X \times X^{\loze}$ be such that $(x_n,x^{\loz}_n)\xrightarrow{bw\times\|\cdot\|_\loz} (x,x^\loz)$, where $(x,x^\loz) \in X\times X^\loze$. It follows from Proposition \ref{loz1} that
$\langle x_n^\loz , \overrightarrow{px_n}\rangle \rightarrow  \langle x^\loz , \overrightarrow{px}\rangle,$
which implies that $\pi_p$ is sequentially $bw\times\text{$\|\cdot\|_\loz$}$-continuous. Now, since convergence in the metric implies weak convergence, we conclude that $\pi_p$ is continuous.
\item[{\rm(ii)}]  Let $x^\loz,y^\loz \in X^\loze$, $x \in X$ and $\alpha, \beta\in\mathbb{R}$. Then
\begin{align*}
\pi_p(x,\alpha x^{\loz}+\beta y^{\loz})&=\langle \alpha x^{\loz}+\beta y^{\loz},\overrightarrow{px} \rangle=\alpha\langle x^{\loz},\overrightarrow{px} \rangle+\beta\langle y^{\loz},\overrightarrow{px} \rangle=\alpha\pi_p(x,x^{\loz})+\beta\pi_p(x,y^{\loz}).
\end{align*}\end{description}
The proof is completed. \hfill\qed

\begin{definition}{\rm \cite[Section 2.2]{Bacak2014}\label{lscepi} Suppose that $(X,d)$ is an Hadamard space and $f : X \rightarrow ]-\infty,\infty]$ be a function.
\begin{enumerate}
\item[(i)] The \textit{domain} of f is defined by
$\dom (f):=\{x \in X : f(x) <\infty \}.$
Moreover, $f$ is called \textit{proper} if $\dom (f)\neq \emptyset$.
\item[(ii)] $f$ is \textit{lower semi-continuous} (briefly \textit{l.s.c.}) if the set $\{x\in X:f(x)\leq \alpha\}$ is closed, for each $\alpha \in \mathbb{R}$.
\item[(iii)] $f$ is \textit{convex} if
\[f((1-\lambda)x \oplus \lambda y) \leq (1-\lambda)f(x)+\lambda f(y),~ \text{for each}~x,y \in X~\text{and}~\lambda \in [0,1].\]
\item[(iv)] $f$ is \textit{strongly convex} with parameter $\kappa >0$ if,
\[f((1-\lambda)x \oplus \lambda y) \leq (1-\lambda)f(x)+\lambda f(y)-\kappa \lambda(1-\lambda)d(x,y)^2,\]
whenever $x,y \in X$ and $\lambda \in [0,1]$.
\item[(v)] When $f$ is proper, an element $x\in X$ is said to be a \textit{minimizer} of $f$, if $f(x)=\inf_{z\in X} f(z)$.
\end{enumerate}
}\end{definition}
The set of all proper, l.s.c. and convex extended real-valued functions on $X$ is denoted by $\Gamma(X)$.

\begin{definition} \label{FDEF}{\rm Let $(X,d)$ be an Hadamard space with linear dual space $X^\loze$ and $f:X\rightarrow ]-\infty,\infty]$ is a function. Define the mapping $\mathbb{I}_f:X \times X^\loze \times X^\loze \rightarrow [-\infty,\infty]$ by
$\mathbb{I}_f(x,x^\loz,y^\loz)=\inf_{y \in X} \big\{f(y)+\pi_y(x,x^\loz+y^\loz)\big\}$.
}\end{definition}
\begin{remark}\label{gtyu} For each $x\in X$ and each $x^\loz,y^\loz,u^\loz,v^\loz \in X^\loze$, we have:
\begin{enumerate}
\item[(i)] $\mathbb{I}_f(x,x^\loz,y^\loz)=\mathbb{I}_f(x,y^\loz,x^\loz)$.
\item[(ii)] $\mathbb{I}_f(x,x^\loz,y^\loz)=\mathbb{I}_f(x,u^\loz,v^\loz)$, provided that $x^\loz+y^\loz=u^\loz+v^\loz$.
\item[(iii)] $\mathbb{I}_f(x,x^\loz,y^\loz)=\mathbb{I}_f(x,\mathbf{0}_{X^\loze}, x^\loz+y^\loz)$.
\end{enumerate}
\end{remark}
\begin{definition} \label{mfd} {\rm Let $f:X \rightarrow ]-\infty,\infty]$ be a function where $X$ is an Hadamard space with the linear dual space $X^\loze$. For any $y^\loz \in X^\loze$, set
\[M^f_{y^\loz}:=\big\{(x,x^\loz) \in X \times X^\loze:~\mathbb{I}_f(x,x^\loz,y^\loz) \geq f(x)\big\}.\]
We use the notation $M^f_{\mathbf{0}}:=M^f_{\mathbf{0}_{X^\loze}}$.
}\end{definition}
\begin{lemma}\label{mtransl}  Let $X$ be an Hadamard space with linear dual space $X^\loze$, $p \in X$ and $y^\loz \in X^\loze$. Then
\[M^f_{y^\loz}=\big\{(x,x^\loz-y^\loz) \in X \times X^\loze:(x,x^\loz) \in M^f_{\mathbf{0}} \big\}=M^{\tilde{f}}_{\mathbf{0}},\]
where,  $\tilde{f}(\cdot):=f(\cdot)-\pi_p(\cdot, y^\loz)$.
\end{lemma}
\begin{proof} By using Remark \ref{gtyu}(i)\&(iii), we get:
\begin{align*}
(x,x^\loz-y^\loz)\in M^f_{y^\loz} &\Longleftrightarrow \mathbb{I}_f(x,x^\loz-y^\loz,y^\loz) \geq f(x)\Longleftrightarrow \mathbb{I}_f(x,x^\loz,\mathbf{0}_{X^\loze}) \geq f(x)\Longleftrightarrow (x,x^\loz) \in M^f_{\mathbf{0}}.
\end{align*}
On the other hand, for each  $z\in X$ and $(x,x^\loz)\in X\times X^\loze$,
\begin{align*}
\tilde{f}(z)+\pi_z(x,x^\loz) -\tilde{f}(x) &=f(z)-\langle y^\loz,\overrightarrow{pz}\rangle+\pi_z(x,x^\loz) - f(x)+\langle y^\loz,\overrightarrow{px}\rangle\\
&=f(z)+\langle y^\loz,\overrightarrow{zx}\rangle+\pi_z(x,x^\loz)-f(x)\\
&=f(z)+\pi_z(x,x^\loz+y^\loz)-f(x).
\end{align*}
Now, by taking the infimum over  $z\in X$, we obtain $\mathbb{I}_{\tilde{f}}(x,x^\loz-y^\loz,y^\loz)-\tilde{f}(x)=\mathbb{I}_f(x,x^\loz,y^\loz)-f(x).$ 
This implies that $M^f_{y^\loz}=M^{\tilde{f}}_{\mathbf{0}}$.
\end{proof}

\section{Flat Hadamard spaces}

 Let $X$ be an Hadamard space with linear dual space $X^\loze$ and $M\subseteq X\times X^\loze$.
The \textit{domain} and \textit{range} of $M$ are defined, respectively, by
\begin{equation*}
\Dom(M):=\{ x\in X : \exists\, x^\loz\in X^\loze ~\text{s.t.}~ (x, x^\loz)\in M\},
\end{equation*}
and
\begin{equation*}
\Ran(M):=\{ x^\loz\in X^\loze : \exists\, x\in X ~\text{s.t.}~ (x, x^\loz)\in M\}.
\end{equation*}

\begin{definition}\label{zsd}{\rm  Let $X$ be an Hadamard space with linear dual space $X^\loze$ and $p\in X$ be fixed. We say that $ M \subseteq X \times X^\loze$ satisfies \textit{$\mathcal{F}_l$-property} if for each $\lambda \in [0,1],\, x^\loz \in\Ran(M)$ and each $x,y \in\Dom(M)$,
\begin{equation}\label{EqWp}
\big \langle x^\loz,\overrightarrow{p((1-\lambda)x\oplus \lambda y)} \big \rangle \leq (1-\lambda)\langle x^\loz,\overrightarrow{px}\rangle +\lambda \langle x^\loz,\overrightarrow{py} \rangle.
\end{equation}
}\end{definition}
\begin{remark} Note that
\begin{enumerate}
\item[(i)] \text{$\mathcal{F}_l$-property} is  introduced and investigated in \cite{MR01} as \text{$\mathcal{W}$-property}.

\item[(ii)] $\mathcal{F}_l$-property is independent of the choice of the point $p$ (see \cite[Proposition 2.2]{MR01}).
\end{enumerate}
\end{remark}

\begin{theorem} \cite[Proposition 2.5]{MR01} \label{frtg}The following statements for an Hadamard space $X$ are equivalent.
\begin{enumerate}
\item[(i)] $X$  is flat.

\item[{(ii)}] $\langle \overrightarrow{x((1-\lambda)x\oplus \lambda y)},\overrightarrow{ab} \rangle=\lambda \langle \overrightarrow{xy},\overrightarrow{ab} \rangle$, for all $a,b,x,y\in X$ and all $\lambda\in[0,1]$.

\item[(iii)] $X\times X^\loze$ has $\mathcal{F}_l$-property.

\item[(iv)] Any subset of $X\times X^\loze$ has $\mathcal{F}_l$-property.
 \end{enumerate}
  \end{theorem}
The following example shows that there exists a relation $M \subseteq X \times X^\loze$, in a non-flat Hadamard space $X$, which doesn't have the $\mathcal{F}_l$-property. Moreover, it is easy to check that in any Hadamard space $(X,d)$ with linear dual space $X^\loze$, for each $(x,x^\loz)\in X\times X^\loze$, the singleton set $\{(x,x^\loz)\}$ has $\mathcal{F}_l$-property.
\begin{example}{\rm (Compare with \cite[Example 2.6]{MR01})} \label{ExMonS}{\rm  Consider the following equivalence relation on $\mathbb{N} \times [0,1]$:
\[(n,t)\sim (m,s) :\Leftrightarrow t=s=0 ~\text{or} ~(n,t)=(m,s).\]
Set $ X:=\frac{\mathbb{N} \times [0,1]}{\sim} $; i.e., $X$ is  the set of all equivalence classes of $\sim$. Let $ d :X \times X \rightarrow \mathbb{R} $ be defined by
\[
d([(n,t)],[(m,s)])=\begin{cases}
|t-s| &n=m,
\\
t+s  & n\neq m.
\end{cases}
\]
The geodesic joining $x=[(n,t)]$ to $y=[(m,s)] $ is defined as follows:
\[ (1-\lambda)x \oplus \lambda y:=\begin{cases}
[(n,(1-\lambda)t-\lambda s)] & 0 \leq \lambda \leq \frac{t}{t+s},
\\[1mm]
[(m,(\lambda-1)t+\lambda s)]   & \frac{t}{t+s} \leq  \lambda \leq 1,
\end{cases}
\]
whenever, $x\neq y$ and vacuously $(1-\lambda)x \oplus \lambda x:=x$. It is known that (see \cite[Example 4.7]{Kakavandi}) $(X,d)$ is an $\mathbb{R}$-tree space. It follows from \cite[Example 1.2.10]{Bacak2014} that $\mathbb{R}$-tree spaces are Hadamard space.
 Let $x=[(2,\frac{1}{2})]$, $ y=[(1,\frac{1}{2})] $, $ a=[(3,\frac{1}{3})] $ and $b= [(2,\frac{1}{2})] $. Then
 \[ (1-\lambda)[(2,\sfrac{1}{2})] \oplus \lambda [(1,\sfrac{1}{2})]=\begin{cases}
[(2,\frac{1}{2}-\lambda)] & 0 \leq \lambda \leq \frac{1}{2},
\\[1mm]
[(1,\lambda-\frac{1}{2})]   & \frac{1}{2} \leq  \lambda \leq 1.
\end{cases}
\]
For each $\lambda\in (0,\frac{1}{2}]$ we obtain $\big\langle \overrightarrow{x((1-\lambda)x\oplus \lambda y)},\overrightarrow{ab} \big\rangle=\bigg\langle \overrightarrow{[(2,\frac{1}{2})][(2,\frac{1}{2}-\lambda)]},\overrightarrow{[(3,\frac{1}{3})][(2,\frac{1}{2})] }\bigg \rangle=-\frac{5}{6}\lambda,$
while
$\lambda\langle \overrightarrow{xy},\overrightarrow{ab}\rangle=-\frac{1}{2}\lambda$.
It follows from Theorem \ref{frtg} that  $(X,d)$ is not a flat Hadamard space.
For each $n\in \mathbb{N}$, set $x_n:=[(n,\frac{1}{2})]$ and  $y_n:=[(n,\frac{1}{n})]$. Now,  define  $M:=\big\{(x_n,[\overrightarrow{y_{n+1}y_n}]):n\in \mathbb{N}\big\}\subseteq X\times X^\loze.$
Take $\lambda={1}/{3}$, $p=[(1,1)] \in X$ and $[\overrightarrow{y_5y_4}] \in \Ran(M)$. Clearly, $(1-\lambda)x_1 \oplus \lambda x_3=[(1,\frac{1}{6})]$ and
$\big \langle[\overrightarrow{y_5y_4}],\overrightarrow{p((1-\lambda)x_1 \oplus \lambda x_3)} \big \rangle=\frac{1}{24},$
while
$\frac{2}{3}\langle [\overrightarrow{y_5y_4}],\overrightarrow{px_1}\rangle +\frac{1}{3} \langle [\overrightarrow{y_5y_4}],\overrightarrow{px_3} \rangle=\frac{1}{40}.$
Therefore, $M $ doesn't have  the $\mathcal{F}_l$-property.}
\end{example}

\begin{lemma} {\rm\cite[Proposition 2.2.17]{Bacak2014}} \label{minconv} Suppose $(X,d)$ is an Hadamard space and $f : X \rightarrow ]-\infty,\infty]$ be  l.s.c. and strongly convex with  $\kappa >0$. Then there exists a unique minimizer $x\in X$ of $f$ and each minimizing sequence converges to $x$. Moreover,
$f(x)+\kappa d(x,y)^2 \leq f(y),~ \textit{for each}~ y\in X.$
\end{lemma}
\begin{lemma}\label{pictsonlyconvex} Let $X$ be an Hadamard space with linear dual space $X^\loze$, $p\in X$ and $y^\loz\in X^\loze$. Then  $X\times\{y^\loz\}$ has $\mathcal{F}_l$-property if and only if $\pi_p(\cdot,y^\loz)$ is convex.
\end{lemma}
\begin{proof} Let $a,b \in X$ and $\lambda \in [0,1]$. Then
\begin{align}\label{Rump1}
\pi_p((1-\lambda)a\oplus\lambda b,y^\loz )=\langle y^\loz,\overrightarrow{p((1-\lambda)a\oplus\lambda b)} \rangle,
\end{align}
and
\begin{align}\label{Rump2}
(1-\lambda)\pi_p(a,y^\loz)+\lambda\pi_p(b,y^\loz)=(1-\lambda) \langle y^\loz,\overrightarrow{pa} \rangle +\lambda \langle y^\loz,\overrightarrow{pb} \rangle.
\end{align}
Now, inequalities \eqref{Rump1},  \eqref{Rump2} and \eqref{EqWp} imply that $X\times\{y^\loz\}$ has $\mathcal{F}_l$-property if and only if $\pi_p(\cdot,y^\loz)$ is convex.
\end{proof}
\begin{proposition} \label{cpthg} Let $(X,d)$ be an Hadamard space with linear dual space $X^\loze$ and $f\in\Gamma(X)$. Let  $p,y \in X$ be fixed and arbitrary and let $y^\loz \in X^\loze$ be such that $X\times\{y^\loz\}$ has $\mathcal{F}_l$-property. Then
\begin{enumerate}
\item[(i)] the mapping $g:X \rightarrow  ]-\infty,\infty]$ defined by $g(x)=f(x)+\pi_p(x,y^\loz)$ is proper, l.s.c. and convex.
\item[(ii)] define the mapping $h:X \rightarrow  ]-\infty,\infty]$ by $h(x)=g(x)+\frac{1}{2}d(x,y)^2$. Then $h$ is proper, l.s.c. and strongly convex with the parameter $\kappa=\frac{1}{2}$. Moreover, $h$ has a unique minimizer $x\in X$ such that $ h(z) \geq h(x)+\frac{1}{2}d(x,z)^2$, for each $z\in X$.
\end{enumerate}
\end{proposition}
\begin{description}
\item[{\rm \textit{Proof}. (i)}]   It is clear that $g$ is proper. Lower semi-continuity of $g$ follows from lower semi-continuity of $f$ and Proposition \ref{pictsconvex}(i).
It follows from Lemma \ref{pictsonlyconvex} and  convexity of $f$ that $g$ is convex.
\item[{\rm(ii)}]  Similar to the proof of (i), $h$ is proper and l.s.c. Let $a,b \in X$ and $\lambda \in [0,1]$. By using the convexity of $g$ and  CN-inequality \eqref{cn}, we get:
\begin{align*}
h(&(1-\lambda)a\oplus \lambda b)=g((1-\lambda)a\oplus \lambda b)+\sfrac{1}{2}d((1-\lambda)a\oplus \lambda b,y)^2\\
&\leq (1-\lambda)g(a)+\lambda g(b)+\sfrac{1}{2}\big((1-\lambda)d(a,y)^2+\lambda d(b,y)^2
-\lambda(1-\lambda)d(a,b)^2\big)\\
&=(1-\lambda)(g(a)+\sfrac{1}{2}d(a,y)^2)+\lambda (g(b)+\sfrac{1}{2}d(b,y)^2)
-\sfrac{1}{2}\lambda(1-\lambda)d(a,b)^2\\
&=(1-\lambda)h(a)+\lambda h(b)-\sfrac{1}{2}\lambda(1-\lambda)d(a,b)^2.
\end{align*}
Thus $h$ is strongly convex with the parameter $\kappa=\frac{1}{2}$. In view of Lemma \ref{minconv}, there exists a unique minimizer $x\in X$ such that $ h(z) \geq h(x)+\frac{1}{2}d(x,z)^2$, for each $z\in X$. This completes the proof.\hfill\qed
\end{description}

\begin{corollary}\label{Corcpthg} Let $(X,d)$ be a flat Hadamard space and $f : X \rightarrow ]-\infty,\infty]$ be proper, l.s.c. and convex. Let  $p,y \in X$ be fixed and arbitrary and $y^\loz \in X^\loze$. Then
\begin{enumerate}
\item[(i)] the mapping $g$, defined in Proposition \ref{cpthg}(i), is proper, l.s.c. and convex.
\item[(ii)]  $h$ (defined in Proposition \ref{cpthg}(ii)) is proper, l.s.c. and strongly convex with the parameter $\kappa=\frac{1}{2}$. Moreover, $h$ has a unique minimizer $x\in X$ such that $ h(z) \geq h(x)+\frac{1}{2}d(x,z)^2$, for each $z\in X$.
\end{enumerate}
\end{corollary}
\begin{proof}
Since $X$ is flat, Theorem \ref{frtg} implies that $X\times X^\loze$ has $\mathcal{F}_l$-property and hence for each $y^\loz\in X^\loze$, $X\times\{y^\loz\}$ as a subset of $X\times X^\loze$ has this property too. Now, Proposition \ref{cpthg} completes the proof.
\end{proof}

\section{Monotonicity from polarity point of view}

The concept of monotone operators in Hadamard spaces is introduced in \cite{KakavandiAmini}. Some properties of monotone operators, their resolvents and proximal point algorithm are discussed in  \cite{ChaipunyaKumam,ZamaniRaeisi,KhatibzadehRanjbar2017}. The notions of monotone sets and maximal monotone sets in Hadamard spaces are introduced in \cite{MR01}.
In this section, fundamental properties of (maximal) monotone sets in Hadamard spaces from polarity point of view are considered. Also,  some important results of  \cite{Martinez-LegazSvaiter} are proved in Hadamard spaces.

\begin{definition} \label{monpol} \rm{ Suppose that $X$ is an Hadamard space with linear dual space $X^\loze$. We say that $(x,x^\loz)\in X\times X^{\loz}$ and $(y,y^\loz) \in X\times X^{\loz}$ are \textit{monotonically related}, if $ \langle x^\loz-y^\loz,\overrightarrow{yx}\rangle\geq0 $ and it is denoted by $(x,x^\loz) \mu (y,y^\loz) $.  It is easy to see that $\mu$ is a reflexive and symmetric relation on $X \times X^\loze$. This motivates that monotonically relatedness  can be defined for a subset $M \subseteq X \times X^\loze $. An element $(x,x^\loz)$ is monotonically related to $M$ if $(x,x^\loz)\mu(y,y^\loz) $, for each $(y,y^\loz) \in M$ and this will be denoted by $(x,x^\loz)\mu M$. Moreover, $M \subseteq X\times X^{\loz}$ is said to be a \textit{monotone set} if every $(x,x^\loz),(y,y^\loz) \in M$ are monotonically related. The \textit{monotone polar} of $M$ is
\[ M^\mu:=\{(x,x^\loz)\in X\times X^{\loz} : (x,x^\loz)\mu M\}.\] \label{pol}}
In addition, we often use the notations $M^{\mu\mu}:=(M^\mu)^\mu$ and $M^{\mu\mu\mu}:=(M^{\mu\mu})^\mu$.
\end{definition}

\begin{example}\label{monloz} Suppose that $(X,d)$ is an Hadamard space with linear dual space $X^\loze$ and $f : X \rightarrow ]-\infty,\infty]$ is an arbitrary function. Let  $(u,u^\loz) \in M^f_{\mathbf{0}}$ and $(v,v^\loz) \in M^f_{\mathbf{0}}$. By the definition of $M^f_{\mathbf{0}}$, we get $\mathbb{I}_f(u,u^\loz,\mathbf{0}_{X^\loze}) \geq f(u)$ and $\mathbb{I}_f(v,v^\loz,\mathbf{0}_{X^\loze}) \geq f(v)$. Hence, for each $y \in X$,
\begin{equation} \label{yu1}
f(y)+\pi_y(u,u^\loz) \geq f(u),
\end{equation}
and
\begin{equation} \label{yv2}
f(y)+\pi_y(v,v^\loz) \geq f(v).
\end{equation}
Now, put $y:=v$ and $y:=u$ in \eqref{yu1} and \eqref{yv2}, respectively, to obtain:
\begin{equation} \label{fd}
\pi_v(u,u^\loz) \geq f(u)-f(v),
\end{equation}
and
\begin{equation} \label{fe}
\pi_u(v,v^\loz) \geq f(v)-f(u).
\end{equation}
Adding inequalities \eqref{fd} and \eqref{fe} and since $\pi_u(v,v^\loz) =-\pi_v(u,v^\loz)$, we get $\pi_v(u,u^\loz)-\pi_v(u,v^\loz)\geq 0,$
this means that $\langle u^\loz-v^\loz,\overrightarrow{vu}\rangle \geq 0$, i.e., $M^f_{\mathbf{0}}$ is monotone. Finally, Lemma \ref{mtransl} implies that for each $y^\loz \in X^\loze$, $M^f_{y^\loz}$ is a monotone set.
\end{example}

\begin{example}\label{ExMonRtree} Let $X$ be the same as in Example \ref{ExMonS}. For each $n\in \mathbb{N}$, set $M:=\big\{\big(x_n,\big[\overrightarrow{x_{n+1}y_n}\big]\big):n\in\mathbb{N}\big\}$, where $x_n=[(n,1)]$ and $y_n=[(n,0)]$. Then $M$ is monotone. Indeed, for each $n,m \in \mathbb{N}$,
\begin{align*}
&\big\langle \big[\overrightarrow{x_{n+1}y_n}\big]-\big[\overrightarrow{x_{m+1}y_m}\big],\overrightarrow{x_mx_n}\big\rangle =\begin{cases}
2,  &~~~~n\in \{m-1,m+1\},\\
0, &~~~~o.w. \\
\end{cases}
\end{align*}
\end{example}

\begin{definition} {\rm Let $X$ be an Hadamard space. A monotone set $M \subseteq X \times X^\loze$ is called \textit{maximal} if there is no monotone set $L \subseteq X \times X^\loze$  that properly contains $M$.
}\end{definition}

\begin{example}
\begin{enumerate}
\item[(i)] Let $X$ and $M$ be the same as in Example \ref{ExMonRtree}. It is shown that $M$ is monotone. We claim that it is not maximal monotone. To see this, let $x=[(n,1)]$, $x^\loz=\overrightarrow{[(n+1,1)][(n,0)]}$,
$z=[(1,0)]$ and $z^\loz=\big[\overrightarrow{[(1,\frac{1}{2})][(1,1)]}\big]$  be arbitrary. Then
$(x,x^\loz)\in M$,  $(z, z^\loz)\notin M$ and
 \begin{align*}
\langle z^\loz-x^\loz, \overrightarrow{xz} \rangle
&= \big\langle\overrightarrow{[(1,\sfrac{1}{2})][(1,1)]}), \overrightarrow{[(n,1)][(1,0)]} \big\rangle-\big\langle\overrightarrow{[(n+1,1)][(n,0)]}), \overrightarrow{[(n,1)][(1,0)]} \big\rangle=\begin{cases}
\frac{1}{2},  &~~~~n=1,\\
\frac{3}{2}, &~~~~n\neq 1. \\
\end{cases}
\end{align*}
This implies that $M$ is not maximal.

\item[(ii)] Suppose that $(X,d)$ is a flat Hadamard space with linear dual space $X^\loze$, $y^\loz \in X^\loze$ and  $f\in\Gamma(X)$.
It follows from Example \ref{monloz} that $M^f_{\mathbf{0}}$ is a monotone set. Let $(y, y^\loz) \in X \times X^\loze \setminus M^f_{\mathbf{0}}$. By using Lemma \ref{mtransl}, we conclude that:
\begin{equation}\label{mji}
(y,{\mathbf{0}}_{X^\loze}) \notin M^g_{\mathbf{0}},
\end{equation}
where $g(z)=f(z)+\pi_p(z,-y^\loz)$, $z \in X$.

Now, consider the extended real-valued mapping $h(z)=g(z)+\pi_z(x,[\overrightarrow{xy}])$, $z\in X$. It follows from Corollary \ref{Corcpthg}(ii) that $h$ has unique minimizer $x \in X$. Hence, for each $z \in X$ we obtain $h(z)+\pi_z(x,\mathbf{0}_{X^\loze})=h(z)\geq h(x)$. Therefore, $\mathbb{I}_h(x,[\overrightarrow{xy}],[\overrightarrow{yx}]) \geq h(x)$. Consequently, by using Lemma \ref{mtransl}, we get
$
(x,[\overrightarrow{xy}]) \in M^h_{[\overrightarrow{yx}]}=M^{\tilde{h}}_{\mathbf{0}}=M^g_{\mathbf{0}},
$
where $\tilde{h}(\cdot)=h(\cdot)+\pi_x(\cdot,[\overrightarrow{xy}])$. Thus
\begin{equation}\label{RMg0}
(x,[\overrightarrow{xy}]) \in M^g_{\mathbf{0}}.
\end{equation}
It follows from \eqref{mji} that $x\neq y$. On the other hand, we derive from Lemma \ref{mtransl} and  \eqref{RMg0} that   $(x,[\overrightarrow{xy}]+y^\loz) \in M^f_{\mathbf{0}}$. Moreover,
$
\langle[\overrightarrow{xy}]+y^\loz-y^\loz, \overrightarrow{yx}\rangle =\langle\overrightarrow{xy},\overrightarrow{yx}\rangle=-d(x,y)^2<0$,
which implies that  $(y, y^\loz)$ is not monotonically related to $M^f_{\mathbf{0}}$. Hence,  $M^f_{\mathbf{0}}$ is a maximal monotone set. Finally, it follows from Lemma \ref{mtransl} and Corollary \ref{Corcpthg}(i) that $M^f_{y^\loz}$ is a maximal monotone set.
\end{enumerate}
\end{example}

The following well-known fact states that every monotone set in Hadamard spaces can be extended to a maximal monotone set. The proof is similar to that proof of {\rm\cite[Theorem 20.21]{Bauschke2017}}. For the sake of completeness, we add a proof.
\begin{proposition}\label{maxex} Suppose that $M \subseteq X \times X^\loze$ is a monotone set. Then there exists a maximal monotone extension (which is not necessarily unique) of $M$; i.e., a maximal monotone set $\widetilde{M}\subseteq X \times X^\loze$ such that $M\subseteq \widetilde{M}$.
\end{proposition}
\begin{proof}There are two cases:
\begin{description}
\item[{\rm Case I:}] $M\neq\emptyset$;
In this case, consider the set ${\mathfrak{M}}:=\{L\subseteq X \times X^\loze: L~ \textit{is a monotone set and } M \subseteq L\}$.
It is clear that  $({\mathfrak{M}},\preceq)$ is a partially ordered set, where for every $L_1 ,L_2 \in {\mathfrak{M}}$, $L_1 \preceq L_2 :\Leftrightarrow L_1\subseteq L_2$. Let ${\mathfrak{A}}$ be a chain in ${\mathfrak{M}}$. One can see that $\cup_{A\in {\mathfrak{A}}} A$ is an upper bound of ${\mathfrak{A}}$. Now, by using  the Zorn's lemma, there exists a maximal element $\widetilde{M} \in \mathfrak{M}$.
\item[{\rm Case II:}]  $M=\emptyset$;
In this case, let $p \in X$ be fixed. By Case I, $\{(p,\mathbf{0}_{X^\loze})\}$ has a maximal monotone extension, say  $\widetilde{M}$. Obviously,   $\widetilde{M}$ is a maximal monotone extension of $M$.
\end{description}
Then in any cases $M \subseteq X \times X^\loze$ has a maximal monotone extension.
\end{proof}

\begin{remark} Let $X$ be an Hadamard space and  $M \subseteq X \times X^\loze$ be a monotone set. In view of  Proposition \ref{maxex}, there exists $\widetilde{M}\subseteq X \times X^\loze$  as a maximal monotone extension of $M$. Set
\[\widetilde{\mathfrak{M}}(M):=\big\{\widetilde{M}\subseteq X \times X^\loze:\widetilde{M}~ \textit{is a maximal monotone extension of M}\big\}.\]
The set of all maximal monotone sets in $X \times X^\loze$ is denoted by ${\mathfrak{MS}}(X)$; on the other words, ${\mathfrak{MS}}(X)=\widetilde{\mathfrak{M}}( X \times X^\loze)$.
\end{remark}
\begin{proposition}
\label{vt} Let $M \subseteq X \times X^\loze $ be a monotone set. Then $(x,x^\loz) \in M^\mu$ if and only if $M\cup \{(x,x^\loz)\}$ is monotone.
\end{proposition}
\begin{proof}
It is straightforward.
\end{proof}
\begin{definition}{\rm{\rm\cite[Definition 2.1]{Koslowski2001}}}\label{Z1} \rm{Let $\mu$ be a relation from $A$ to $B$. Define two functions $\sigma:\mathcal{P}(A)\rightarrow\mathcal{P}(B)$ and $\tau:\mathcal{P}(B) \rightarrow\mathcal{P}(A)$ as follows: \begin{align*}
\sigma(U)&=\{b\in B:u\mu b,\forall u\in U\},~~
\tau(V)=\{a\in A:a\mu v,\forall v\in V\}.
\end{align*}
Then  $ (\sigma,\mathcal{P}(A),\mu,\mathcal{P}(B),\tau)$ or simply $(\sigma,\mu,\tau)$ is called a \textit{polarity}.}\end{definition}
\begin{lemma}{\rm{\rm\cite[Proposition 2.4]{Koslowski2001}}} \label{Z2} Let  $(\sigma,\mathcal{P}(A),\mu,\mathcal{P}(B),\tau)$  be a polarity. Then
\begin{enumerate}
\item[(i)] $\mathcal{P}(A)$ and $\mathcal{P}(B)$ are partially ordered sets, ordered by set inclusion, and $\sigma$ and $\tau$ are order reversing functions.
\item[(ii)] $\sigma\tau$ and $\tau\sigma$ are order increasing; i.e., \begin{align*}
U\subseteq \tau\sigma(U) ~\text{and}~ V\subseteq \sigma\tau(V),~ \forall  U\subseteq A, \forall V\subseteq B.
\end{align*}
\item[(iii)] $\tau$ is a quasi-inverse for $\sigma$ and $\sigma$ is a quasi-inverse for $\tau$; i.e., $\sigma\tau\sigma=\sigma$ and $\tau\sigma\tau=\tau$.
\end{enumerate}
\end{lemma}

\begin{definition}{\rm{\rm\cite[Definition 5.1]{BurrisSakappanavar1981}}} \label{Z3} {\rm A function $\sigma:\mathcal{P}(A)\rightarrow \mathcal{P}(A)$ is a \textit{closure operator} on a set $A$, if $\sigma$ has the following properties: \begin{enumerate}
\item[(i)] $\sigma(U) \subseteq \sigma(V)$,  $\forall\,  U,V\subseteq A$~\text{with}~$U\subseteq V$.
\item[(ii)] $\sigma(\sigma(U))=\sigma(U)$, $\forall\,  U\subseteq A$.
\item[(iii)]  $U \subseteq \sigma(U)$, $\forall\,  U\subseteq A$.
\end{enumerate}
}
\end{definition}
\begin{proposition}\label{yum}  Let $X$ be an Hadamard space and $M \subseteq X \times X^\loze$. Then
\begin{enumerate}
\item[(i)]  $M \subseteq M^{\mu\mu}$.
\item[(ii)] $M^{\mu\mu\mu}=M^\mu$.
\item[(iii)] $M_1\subseteq M_2 \Rightarrow M_2^\mu \subseteq M_1^\mu$, ~$\forall\, M_1,M_2 \subseteq X \times X^\loze$.
\end{enumerate}
\end{proposition}
\begin{proof}
Consider the mapping $\zeta :\mathcal{P}(X \times X^\loze) \rightarrow \mathcal{P}(X \times X^\loze),~ M \mapsto M^\mu$.  
It follows from Definition \ref{Z1} that $(\zeta,\mu,\zeta)$, is a polarity. Consequently,  the items (i), (ii) and (iii) follow from Lemma \ref{Z2}.
\end{proof}
\begin{proposition}\label{pmu} Let $(X,d)$ be an Hadamard space and $\{M_i\}_{i\in I} \subseteq \mathcal{P}(X \times X^\loze)$ be a family of monotone sets. Then
\begin{enumerate}
\item[(i)] $(\bigcup_{i \in I} M_i)^\mu=\bigcap_{i\in I}M_i^\mu.$
\item[(ii)] $\emptyset^{\mu}=X\times X^{\loze}$.
\item[(iii)] $(X\times X^{\loze})^{\mu}=\emptyset$,~\text{provided that $\card(X)>1$}.
\end{enumerate}
\end{proposition}
\begin{description}
\item[{\rm \textit{Proof.} (i)}]  By using  Definition \ref{pol}, for each $(x,x^\loz)$, \begin{align*}
(x,x^\loz) \in \big(\bigcup_{i \in I} M_i\big)^\mu &\Longleftrightarrow (x,x^\loz)\mu  \big(\bigcup_{i \in I} M_i \big)\Longleftrightarrow (x,x^\loz)\mu M_i  ~,\forall i \in I\Longleftrightarrow (x,x^\loz) \in M_i^\mu ,\forall i \in I\\&\Longleftrightarrow (x,x^\loz) \in \bigcap_{i\in I}M_i^\mu.
\end{align*}
\item[{\rm(ii)}] Clearly $\emptyset^{\mu}\subseteq X\times X^{\loze}$. Conversely, assume to the contrary that  there exists $ (x,x^\loz) \in X\times X^{\loze}\setminus \emptyset^{\mu}$. Hence, there is $(y,y^\loz)\in \emptyset$ such that $(x,x^\loz) $ and $(y,y^\loz)$ are not monotonically related, yields a contradiction.
\item[{\rm(iii)}] It is enough to prove that $(X\times X^{\loze})^{\mu}\subseteq\emptyset$. Let $ (x,x^\loz) \in (X\times X^{\loze})^{\mu}$ and let $a\in X$ with $a\neq x$.  Clearly $(a, x^\loz-[\overrightarrow{xa}])\in X\times  X^{\loze}$ and so $(x,x^\loz)\mu(a, x^\loz-[\overrightarrow{xa}])$. Then
\[
0\leq\langle x^\loz-(x^\loz-[\overrightarrow{xa}]), \overrightarrow{ax}\rangle=\langle [\overrightarrow{xa}], \overrightarrow{ax}\rangle=\langle \overrightarrow{xa}, \overrightarrow{ax}\rangle=-d(x,a)^2
,\]
a contradiction. \hfill\qed
\end{description}

\begin{definition}\label{muc} {\rm For each $M \subseteq X \times X^\loze $, the \textit{closure operator} induced by the polarity $\mu$ is defined by the mapping $M\mapsto M^{\mu\mu}$. Moreover, we say that $M^{\mu\mu}$ is  the \textit{$\mu$-closure} of $M$ and $M$ is $\mu$-closed if $M^{\mu\mu}=M$.}
\end{definition}
\begin{remark} {\rm It follows from Proposition \ref{yum}(ii) that the family of all $\mu$-closed sets is the family of polars $\{A^\mu: A\subseteq X\times X^\loze\}$. One can see that $M^{\mu\mu}$ is the smallest $\mu$-closed set containing $M$.}
\end{remark}
\begin{proposition}\label{mtg}  Let $X$ be an Hadamard space and $M \subseteq X \times X^\loze$. Then the following statements are equivalent:
\begin{enumerate}
\item[(i)] $M$ is monotone.
\item[(ii)] $M\subseteq M^\mu$.
\item[(iii)]  $M^{\mu\mu} \subseteq M^\mu$.
\item[(iv)] $M^{\mu\mu}$ is monotone.
\end{enumerate}
In addition, $M \in {\mathfrak{MS}}(X)$ if and only if $M= M^\mu$. Moreover, every element of $\mathfrak{MS}(X)$ is $\mu$-closed.
\end{proposition}
\begin{description}
\item[{\rm\textit{Proof.}  (i)$\Leftrightarrow$(ii)}]    Clearly, $M$ is monotone if and only if every two members of $M$ are monotonically related, or equivalently $M\subseteq M^\mu$. Hence (i) and (ii) are equivalent.

\item[{\rm(ii)$\Rightarrow$(iii)}] An immediate consequence of Proposition  \ref{yum}(iii).

\item[{\rm(iii)$\Rightarrow$(iv)}] By assumption  and Proposition \ref{yum}(iii),  we obtain
 $M^{\mu\mu}=(M^\mu)^\mu\subseteq  (M^{\mu\mu})^{\mu}$.
 Now, equivalence of (i) and (ii), implies that $M^{\mu\mu}$ is monotone.
\item[{\rm(iv)$\Rightarrow$(ii)}] Let $M^{\mu\mu}$ be a monotone set. By using Proposition \ref{yum}(i),  applying (i)$\Rightarrow$(ii) to $M^{\mu\mu}$ and  Proposition \ref{yum}(ii), we get $M\subseteq M^{\mu\mu} \subseteq M^{\mu\mu\mu}=M^\mu$.
Therefore, $M$ is monotone by (ii)$\Rightarrow$(i).

We know that $M$ is monotone if and only if $M \subseteq M^\mu$. Moreover, maximality of $M$ is equivalent to $M^\mu\subseteq M$. Therefore, $M \in \mathfrak{MS}(X)$ if and only if $M= M^\mu$. Finally, by using this fact, for each $M \in \mathfrak{MS}(X)$ we have
 $M= M^\mu=(M^\mu)^\mu=M^{\mu\mu}$.
Now, it follows from Definition \ref{muc} that $M$ is a $\mu$-closed set.\hfill{\qed}

\end{description}

\begin{proposition}\label{bty} Let $X$ be an Hadamard space and $M \subseteq X \times X^\loze$ be a monotone set. Then the following hold:
\begin{enumerate}
\item[(i)] $M^\mu=\bigcup_{\widetilde{M}\in \mathfrak{M}(M)}\widetilde{M}$.
\item[(ii)]  $M^{\mu\mu}=\bigcap_{\widetilde{M}\in \mathfrak{M}(M)}\widetilde{M}$.
\end{enumerate}
\end{proposition}
\begin{description}
\item[{\rm\textit{Proof.}  (i)}]  Let $(x,x^\loz)\in M^\mu$ be given. It follows from Proposition \ref{vt} that $M\cup\{(x,x^\loz)\}$ is a monotone set. By Proposition \ref{maxex}, there exists maximal monotone extension $\widetilde{M}$ for $M\cup\{(x,x^\loz)\}$. Hence,  $\widetilde{M} \in \mathfrak{M}(M)$ and $(x,x^\loz)\in \widetilde{M}$. Therefore, $M^\mu\subseteq\bigcup_{\widetilde{M}\in \mathfrak{M}(M)}\widetilde{M}$. Conversely, let $(x, x^\loz)\in\bigcup_{\widetilde{M}\in \mathfrak{M}(M)}\widetilde{M}$. Then there exists $\widetilde{M}\in \mathfrak{M}(M)$ such that $(x,x^\loz)\in \widetilde{M}$. By  using Proposition \ref{mtg} and Proposition \ref{yum}(iii),  we get $(x,x^\loz)\in \widetilde{M}=\big(\widetilde{M}\big)^\mu\subseteq M^\mu$. Consequently, $\bigcup_{\widetilde{M}\in \mathfrak{M}(M)}\widetilde{M}\subseteq M^\mu$.

\item[{\rm(ii)}] By using (i), Proposition \ref{pmu}(i) and Proposition \ref{maxex}, we obtain: \[M^{\mu\mu}=(M^\mu)^\mu=\bigg(\bigcup_{\widetilde{M}\in \mathfrak{M}(M)}\widetilde{M}\bigg)^\mu=\bigcap_{\widetilde{M}\in \mathfrak{M}(M)}\big(\widetilde{M}\big)^\mu=\bigcap_{\widetilde{M}\in \mathfrak{M}(M)}\widetilde{M}.\]
 \end{description}
We are done.\hfill\qed
\begin{lemma} Let $X$ be an Hadamard space and $M \subseteq X \times X^\loze$. Then $M$ is monotone if and only if  there exists $\widetilde{M}\in\mathfrak{MS}(X)$ such that $\widetilde{M}\subseteq M^\mu$.
\end{lemma}
\begin{proof} Let $M$ be a monotone set. It follows from Proposition \ref{bty}(i) that $M^\mu$ contains a maximal monotone set. Conversely, let there exists $\widetilde{M}\in\mathfrak{MS}(X)$ such that $\widetilde{M}\subseteq M^\mu$. Now, by using Proposition  \ref{yum}(i)$\&$(iii)  and Proposition \ref{mtg}, we conclude that $M \subseteq(M^\mu)^\mu\subseteq\big(\widetilde{M}\big)^\mu=\widetilde{M}\subseteq M^\mu$.
The claim therefore follows from Proposition \ref{mtg}((i)$\Leftrightarrow$(ii)).
\end{proof}

\begin{proposition} \label{mumon} Let $X$ be an Hadamard space and $M \subseteq X\times X^\loze$ be a maximal monotone set. Then \begin{enumerate}
\item[(i)] for each $u\in X$, $M^\loz_u:=\{ u^\loz\in X^\loze: (u,u^\loz)\in M\}$ is a closed and convex subset of $X^\loze$.
\item[(ii)] $M$ is sequentially $bw\times${\scalebox{0.9}{$\|\cdot\|_\loz$}}-closed in $X\times X^\loze$, (and hence $d \times \|\cdot\|_\loz$-closed).
\item[(iii)] If $\Dom(M) \subseteq X$ is bounded, then $M$ is sequentially weakly$\times \|\cdot\|_\loz$-closed in $X\times X^\loze$.
\end{enumerate}
\end{proposition}
\begin{description}
\item[{\rm\textit{Proof. }(i)}]  Let $u^\loze\in\overline{M^\loz_u}$.  Then there exists $\{u_n^\loz\} \subseteq M^\loz_u$ such that $u_n^\loze\xrightarrow{\|\cdot\|_\loz} u^\loz$.
For each  $(x,x^\loz) \in M$, by monotonicity of $M$, we have \mbox{$\langle u_n^\loz-x^\loz,\overrightarrow{xu}\rangle\!\geq\!0$.} By applying Proposition \ref{loz1}, as $n\rightarrow\infty$, we get  $\langle u^\loz-x^\loz,\overrightarrow{xu} \rangle \geq 0$. Therefore
$(u,u^\loz)\in M^\mu$. Now, maximality of $M$ implies  $(u,u^\loz)\in M$ or equivalently $u^\loz\in M^\loz_u$. Consequently, $M^\loz_u$ is closed. For proving convexity of $M^\loz_u$, let $u^\loz, v^\loz\in M^\loz_u$ and $\lambda\in [0, 1]$ be arbitrary and fixed. Then for each  $(x,x^\loz) \in M$, by monotonicity of $M$, we have:
 \[ \langle (1-\lambda)u^\loz+\lambda v^\loz -x^\loz,\overrightarrow{xu} \rangle=(1-\lambda) \langle u^\loz-x^\loz,\overrightarrow{xu} \rangle +\lambda \langle v^\loz-x^\loz,\overrightarrow{xu} \rangle \geq 0.\]
 Consequently, $(u, \lambda u^\loz +(1-\lambda)v^\loz) \in M^\mu$. Again,  maximality of $M$ implies that $(u, \lambda u^\loz +(1-\lambda)v^\loz) \in M$; i.e.,
 $\lambda u^\loz +(1-\lambda)v^\loz\in M^\loz_u$ and so $M^\loz_u$ is convex.

\item[{\rm(ii)}] Let $\{(x_n,x^\loz_n)\} \subseteq M$ be a sequence such that $(x_n, x^\loz_n)\xrightarrow{bw\times\|\cdot\|_\loz}(x, x^\loz)$. For each $(y,y^\loz) \in M$, $\langle x_n^\loz-y^\loz,\overrightarrow{yx_n} \rangle \geq 0$.
It follows from Proposition \ref{loz1} that $\langle x^\loz-y^\loz,\overrightarrow{yx} \rangle=\lim_{n\rightarrow+\infty}\langle x_n^\loz-y^\loz,\overrightarrow{yx_n} \rangle \geq 0$. Hence $(x,x^\loz) \in M^\mu=M$. Thus $M$ is sequentially $bw\times${$\|\cdot\|_\loz$}-closed in $X\times X^\loze$. Finally,  since any convergent sequence in a metric space is bounded and weakly convergent, it follows that $M$ is $d \times \|\cdot\|_\loz$-closed.
\item[{\rm(iii)}] It is an immediate consequence of (ii). \hfill\qed
\end{description}


\end{document}